\documentclass[12pt,a4paper,leqno]{article}
\usepackage{makeidx}
\makeindex

\usepackage{latexsym}
\usepackage{latexsym,bbm}
\usepackage{amssymb}
\usepackage{amsmath}
\usepackage{amsopn}
\usepackage{epsfig}
\usepackage{psfrag}

\usepackage[mathscr]{eucal}
\usepackage{multicol}

\usepackage{titlesec}
\usepackage{etoolbox}
\makeatletter
\patchcmd{\ttlh@hang}{\parindent\z@}{\parindent\z@\leavevmode}{}{}
\patchcmd{\ttlh@hang}{\noindent}{}{}{}
\makeatother

\titleformat{\section}[hang]
  {\normalfont\sffamily\bfseries}
  {\thesection.}{.5em}{}

\usepackage{amsthm}
\newtheoremstyle{statements}%
{\bigskipamount}%Space above
{\bigskipamount}%Space below
{\itshape}%Body font
{}%Indent amount
{\bfseries\sffamily}%Theorem head font
{.}%Punctuation after theorem head
{.5em}%Space after theorem head
{}%Theorem head spec (can be left empty, meaning `normal
\newtheoremstyle{definitions}%
{\bigskipamount}%Space above
{\bigskipamount}%Space below
{}%Body font
{}%Indent amount
{\bfseries\sffamily}%Theorem head font
{.}%Punctuation after theorem head
{.5em}%Space after theorem head
{}%Theorem head spec (can be left empty, meaning `normal\thmname \thmnumber \thmnote

\theoremstyle{statements}
\newtheorem{thm}{Theorem}[section]
\newtheorem{lem}[thm]{Lemma}

\newtheorem{cor}[thm]{Corollary}
\newtheorem*{thm*}{Theorem}
\newtheorem*{lem*}{Lemma}
\newtheorem*{ques*}{Question}
\newtheorem*{conj*}{Conjecture}
\newtheorem*{prob*}{Problem}
\newtheorem*{prop*}{Statement}
\newtheorem*{cor*}{Corollary}

\theoremstyle{definitions}
\newtheorem{df}[thm]{Definition}

\newtheorem{rems}[thm]{Remarks}

\newtheorem*{df*}{Definition}
\newtheorem*{rem*}{Remark}
\newtheorem*{exple*}{Example}

\DeclareMathOperator{\cdim}{cdim}
\DeclareMathOperator{\conv}{conv}
\DeclareMathOperator{\bd}{bd}

\begin{document}

\title{\normalfont\sffamily\large\bfseries On the represetation of finite convex geometries\\ with convex sets}
\author{\it J. Kincses\thanks{Research supported by the ERC Advanced Research Grant
no 267165 (DISCONV).}}
%\date{\today}
\date{}

\maketitle
\begin{abstract}
\noindent Very recently Richter and Rogers proved that any convex geometry 
can be represented by a family of convex polygons in the plane.
We shall generalize their construction and obtain
a wide variety of convex shapes for representing convex geometries.
We present an Erd\H os-Szekeres type obstruction, which answers a question of 
Cz\'edli negatively, that is general convex geometries
cannot be represented with ellipses in the plane.
Moreover, we shall prove that one cannot even bound
the number of common supporting lines of the pairs of the representing convex sets.
In higher dimensions we prove that all convex geometries can be represented 
with ellipsoids.
\end{abstract}

\section{Introduction}

Finite convex geometries were introduced by Edelman and Jamison in \cite{{edelmanjamison}}
and these structures were intensively studied in combinatorics and
in lattice theory (\cite{kortelovasz}, \cite{adarichevanation}).
They are an abstraction of geometric convexity in affine spaces
and lately various results were born about the geometric representation of convex geometries.
In this direction Kashiwabara, Nakamura and
Okamoto \cite{kashiwabara} obtained a basic result proving that
any convex geometry can be represented as a
''generalized convex shelling'' in some affine spaces, meaning that
there is an embedding in ${\mathbb R}^d$ so that each set in the geometry
is convex if and only if its embedding is convex with respect to a fixed
external set of points.

Cz\'edli proved in \cite{czedli} that convex geometries of convex dimension at most 2
may be represented as a set of circles in the plane.
This result initiated a new research area: representing
convex geometries with ''nice'' convex sets.
Very recently Richter and Rogers \cite{richterrogers} proved that any convex geometry 
can be represented by a family of convex polygons in the plane.
Cz\'edli, Kincses \cite{czedlikincses} represented convex geometries by almost-circles.
In Section 3 we generalize the construction of Richter and Rogers and we will obtain
a wide variety of convex shapes for representing convex geometries.

Adaricheva and Bolat \cite{adarichevabolat} found an obstruction for
representing general convex geometries with circles in the plane. In Section 4
we present an Erd\H os-Szekeres type obstruction, which answers a question of 
Cz\'edli \cite{czedli} negatively, that is general convex geometries
cannot be represented with ellipses.
Moreover, in Section 5 we shall prove that one cannot even bound
the number of common supporting lines of the pairs of the representing convex sets.
This will be obtained by giving an upper bound for the convex dimension of the geometry
in terms of the number of common supporting lines of the pairs.

Adaricheva and Cz\'edli \cite{adarichevabolat} raised the question
whether every finite convex geometry can be represented 
with balls in ${\mathbb R}^d$. In Section 6 we prove that this can be done
with ellipsoids.

\section{Preliminaries}

There are several equivalent definitions of finite convex geometries.
We shall use three of them.

\begin{df}
Let $E$ be a finite set. A convex geometry on $E$ is a collection
$\cal C$ of subsets (called convex sets) of $E$ with the following properties.
\begin{enumerate}
\item $\emptyset$ and $E$ are in $\cal C$,
\item If $X, Y\in{\cal C}$ then $X\cap Y\in {\cal C}$,
\item If $X\in {\cal C}\setminus \{E\}$ then there is $e\in E\setminus X$ so
$X\cup \{e\}\in {\cal C}$.
\end{enumerate}
\end{df}

\noindent Convex geometries ${\cal C}_1$ on $E_1$ and ${\cal C}_2$ on $E_2$ are isomorphic if there is a bijection $\varphi\colon E_1\to E_2$ such that
$\varphi(X)\in {\cal C}_2$ if and only if $X\in {\cal C}_1$.

In a convex geometry the convex closure of a set can be defined as the intersection
of the convex sets containing the given set. This operator has nice properties and
it serves an equivalent definition of convex geometries (\cite{adarichevanation}).

\begin{df}
A pair $(E,\Phi)$ is a convex geometry, if it satisfies the following properties:
\begin{enumerate}
\item $E$ is a set, called the set of points, and $\Phi\colon 2^E\to 2^E$ is a closure
operator, that is, for all $X\subseteq Y \subseteq E$, we have
$X\subseteq\Phi(X)\subseteq\Phi(Y) =\Phi(\Phi(Y ))$.
\item If $A\subseteq E$, $x, y\in E\setminus\Phi(A)$, and $\Phi(A\cup {x}) =\Phi(A\cup {y})$, then $x = y$.
(This is the so-called anti-exchange property.)
\item $\Phi(\emptyset) = \emptyset$.
\end{enumerate}
\end{df}

The next way to define a convex geometry on $E$ uses a collection of
orderings (\cite{edelmanjamison}),
denoted by $\preccurlyeq_i$. Throughout the paper, orderings are total and antisymmetric.

\begin{df}
We say that $(E,\cal C)$ is generated by a family $\{\preccurlyeq_i\}_{i=1}^m$
of orderings on $E$ if
\[{\cal C} = \{\emptyset\}\cup \{X\subseteq E\colon \forall y\not\in X,
\ \exists i\text{ so that }\forall x\in X,\ x\prec_i y\}\]
\end{df}

For a convex geometry, its {\it convex dimension}, denoted by $\cdim(E,{\cal C})$, is defined
as the smallest number of orderings which generate the geometry (see \cite{edelmansaks}).

The usual convex hull in ${\mathbb R}^d$ will be denoted by $\conv_{{\mathbb R}^d}$.
For a finite family ${\mathscr K}$ of compact convex sets  in ${\mathbb R}^d$
there is a natural closure operator derived from the affine convex hull
and which is denoted by $\conv_{\mathscr K}$. For any subset 
${\mathscr S}\subseteq {\mathscr K}$ let
\[\conv_{\mathscr K}({\mathscr S})=\{K\in{\mathscr K}\colon K\subseteq
\conv_{{\mathbb R}^d}(\cup_{S\in{\mathscr S}} S). \]
In general this closure will not determine a convex geometry (the critical point
is the anti-exchange property). In the planar case we present a sufficient (but
not necessary) condition which guarantees the anti-exchange property
and in a sense it is more general than that considered in \cite{adarichevabolat}.
For a compact convex set $K$ in the plane the line $l$ is a supporting line if
$K\cap l\not=\emptyset$ but $K$ is contained in one of the halfplanes of $l$.
The line $l$ is a common supporting line of the convex sets $K$ and $L$ if 
it is a supporting line of both $K$ and $L$ and both sets are in the same 
halfplane of $l$.

\begin{lem}\label{fincross}
Let ${\mathscr K}$ be a finite family of compact convex sets in the plane
and suppose that any two of them have at most finitely many common
supporting lines. Then $({\mathscr K},\conv_{\mathscr K})$ is a convex geometry.
\end{lem}

For the proof we need some well known properties of the support function of
a convex set (see \cite{schneider}). For a compact convex set
$K\subseteq {\mathbb R}^d$ its support function is defined as
\[h(K,x)=\max_{y\in K}\langle x,y\rangle\] 
($\langle x,y\rangle$ is the scalar product of $x,y$). The properties we shall use:

(a) $h(K,x)$ is a continuous function on ${\mathbb R}^d$,

(b) for any convex sets $K,L$, $K\subseteq L$ if and only if 
$h(K,x)\le h(L,x)$ for all $|x|=1$,

(c) for any convex sets $\{K_i\}_{i=1}^n$ we have that
\[h(\conv_{{\mathbb R}^d}(\cup_i K_i),x)=\max_i\{ h(K_i,x)\}\]

(d) the line $\langle y,x\rangle=h(K,x)$ is a supporting line of $K$ for any $x\not=0$.

\begin{proof}[Proof of Lemma \ref{fincross}]
Obviously, it is enough to prove the anti-exchange property.
Suppose that ${\mathscr S}\subseteq{\mathscr K}$ and 
$K,L\in {\mathscr K}\setminus\conv_{\mathscr K}({\mathscr S})$
satisfy the conditions of the anti-exchange property.
Using properties (a), (b) and (c) above we have that
there exists an interval $I$ of the unit circle such that
for any $x\in I$
\begin{align*}
h(K,x)&>\max_{M\in{\mathscr S}}\{h(M,x)\}\\
\max\{h(K,x),\max_{M\in{\mathscr S}}\{h(M,x)\}\}&=
\max\{h(L,x),\max_{M\in{\mathscr S}}\{h(M,x)\}\}
\end{align*}
This implies that $h(K,x)=h(L,x)$ for $x\in I$. By property (d), the sets $K$ and $L$
have infinitely many common supporting lines so $K=L$.
\end{proof}
\begin{rems}
1. Lemma \ref{fincross} gives several classes of convex sets to represent
convex geometries. For example families of circles, ellipses, or in general,
nonsingular convex algebraic curves or convex curves with analytic support function
or polygons with distinct vertices all satisfy the condition of the lemma.

2. What we really used in the proof is that for any two of the convex sets the
supporting functions cannot agree on an open interval of the circle.
This is a weaker condition but we think it is more technical than that we stated.
\end{rems}
\section{Representation in the plane}\label{genplane}

Richter and Rogers \cite{richterrogers} proved that any convex geometry 
can be represented by a family of convex polygons in the plane.
In this section we shall modify their construction and
obtain a wide variety of convex shapes for representing convex geometries.

Consider a convex geometry $E$ defined by a collection of orderings
$\{\preccurlyeq_i\}_{i=1}^m$, $m\ge2$. For $x\in E$ let $j_i(x)$ be the
$x$'s place according to the $i^{\text{th}}$ ordering.
Choose the unit vectors $v_i=(\cos(2\pi i/m),\sin(2\pi i/m))$
pointing to the vertices of a regular $m$-gon $R_m$ and 
let $\varepsilon$ be a parameter of the construction where
$0<\varepsilon<|\sec(2\pi/m)|-1$.
Define the points
\[F_i^1(x)=\left(1+\frac{2j_i(x)-1}{2m}\varepsilon\right)v_i\qquad
F_i^2(x)=\left(1+\frac{2j_i(x)}{2m}\varepsilon\right)v_i.\]
The bounds on $\varepsilon$ guarantees that 
\begin{equation}\label{line}
\parbox{11cm}{for any point 
$P=F_i^1(x)$ or $F_i^2(x)$ the line through $P$ with normal $v_i$
contains all the points $F_k^1(y),F_k^2(y)$ $(k\ne i, y\ne x)$
on the same open halfplane of it.}
\end{equation}
This implies that for any $x\in E$ the points $\{F_i^1(x)\}_{i=1}^m$ 
and $\{F_i^2(x)\}_{i=1}^m$ are the vertices of the convex polygons
\[P^1(x)=\conv_{{\mathbb R}^2}\{F_1^1(x),\ldots,F_m^1(x)\}\quad\text{resp.}\quad 
P^2(x)=\conv_{{\mathbb R}^2}\{F_1^2(x),\ldots,F_m^2(x)\}.\]
Finally, for each $x\in E$ choose a compact convex set $K(x)$ such that
$P^1(x)\subseteq K(x)\subseteq P^2(x)$ and let ${\mathscr K}=\{K(x)\}_{x\in E}$
(see Fig \ref{approx}).
It is clear from the construction that
\begin{equation}\label{disjoint}
\parbox{11cm}{if $x\ne y$ then the segments $F_i^1(x)F_i^2(x)$ and
$F_i^1(y)F_i^2(y)$ are disjoint.}
\end{equation}

\begin{figure}[h]
\centering
\includegraphics[width=90mm]{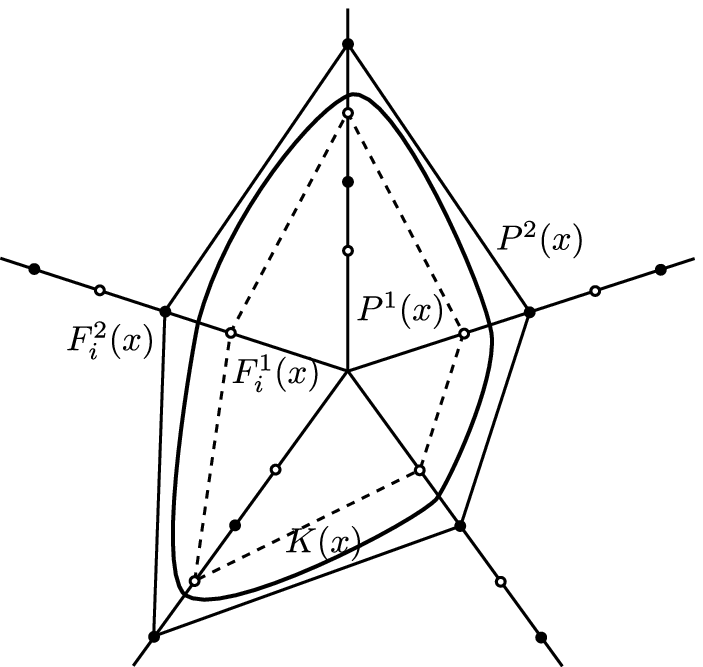}
\caption{}\label{approx}
\end{figure}

\begin{thm}
$({\mathscr K},\conv_{\mathscr K})$ is a finite convex geometry
and the map $\Phi\colon x\to K(x)$ is an isomorphism between the
convex geometries $(E,\{\preccurlyeq_i\}_{i=1}^m)$ and $({\mathscr K},\conv_{\mathscr K})$.
\end{thm}
\begin{proof}
We have to prove that $X\subseteq E$ is convex in $E$ if and only if $\Phi(X)$ is convex 
in ${\mathscr K}$.

Suppose that $X\subseteq E$ is not convex in $E$. Then there is a $z\not\in X$
such that for each $i$, there exists $x_i\in X$ with $z\preccurlyeq_i x_i$.
Thus $j_i(z)<j_i(x_i)$ and $|F_i^2(z)|<|F_i^1(x_i)|$ for all $i$. From the construction
easily follows that 
\[K(z)\subseteq P^2(z)\subseteq \conv_{{\mathbb R}^2}(F_1^1(x_1),\ldots,F_m^1(x_m))
\subseteq\conv_{{\mathbb R}^2}(K(x_1)\cup\ldots\cup K(x_m)).\]
But $\Phi(z)\not\in \Phi(X)$ and $K(z)\in\conv_{\mathscr K}\Phi(X)$ implies
that $\Phi(X)$ is not convex in ${\mathscr K}$.

Conversely, suppose that $X$ is convex in $E$ and let $z\not\in X$.
Then there is an $i$ such that for all $x\in X$ we have
$x\preccurlyeq_i z$ and therefore $j_i(x)<j_i(z)$.
Property \eqref{disjoint} implies that $|F_i^2(x)|<|F_i^1(z)|$ and from
\eqref{line} we have that the line through the point $F_i^1(z)$ with normal $v_i$
contains all the polygons $P^2(x)$ in one of its open halfplane.
This means that $F_i^1(z)\not\in\conv_{{\mathbb R}^2}(\cup_{x\in X}P^2(x))$
and using that $K(x)\subseteq P^2(x)$ we have that
$K(z)\not\in\conv_{\mathscr K}(\Phi(X))$ that is $\Phi(X)$ is convex in ${\mathscr K}$.
\end{proof}
\begin{rems}
1. In our construction each set $K(x)$ satisfies $R_m\subseteq K(x)\subseteq(1+\varepsilon)R_m$
where $R_m$ denotes the regular $m$-gon with vertices $v_i$.
This means that the sets $K(x)$ are arbitrary close to the regular $m$-gon
in the sense of Hausdorff metric (see \cite{schneider}).

The construction is sensitive for the number of orderings. Suppose that $E$ is 
defined by the orderings $\{\preccurlyeq_i\}_{i=1}^m$.
We can define a new set of orderings $\{\curlyeqprec_k\}_{k=1}^{sm}$
in which each of the original orderings occurs $s$ times, that is 
$\curlyeqprec_{lm+t}=\preccurlyeq_{l+1}$, $0\le l\le m-1$, $1\le t\le s$.
The new set obviously generates the same convex geometry but the 
sets $K(x)$ will be close to the regular $sm$-gon which is arbitrarily
close to the unit circle if $s$ is large enough.

2. The construction provides a great freedom in choosing the sets $K(x)$.
Using the approximation results of convex geometry (see \cite{gruber})
various ''nice'' convex shapes can be used.
From our point of view especially important classes are 

$\bullet$ convex sets with analytic support function,

$\bullet$ convex sets with algebraic support function,

\noindent 
because these classes satisfy the condition of Lemma \ref{fincross}.

Here we present a folklore example of a semi-algebraic set approximating
a convex polygon because its defining polynomial has degree equal to the
number of vertices of the polygon, which, in our construction, equal to
the convex dimension of the convex geometry. Let the convex polygon
given by the intersection of halfplanes
\[P=\cap_{i=1}^m\{a_{i}x+b_{i}x-c_{i}\ge 0\}\]
and define the set
\[K_P=\{(x,y)\in P\colon\prod_{i=1}^m(a_{i}x+b_{i}y-c_{i})\ge\alpha\}.\]
For sufficiently small positive $\alpha$ the set $K_P$ approximates $P$
and it is a strictly convex compact set with regular algebraic boundary curve
\[\bd K_P=\{(x,y)\in P\colon\prod_{i=1}^m(a_{i}x+b_{i}y-c_{i})=\alpha\}.\]
\end{rems}

\section{Erd\H os-Szekeres type obstructions}\label{sec:}

Cz\'edli proved in \cite{czedli} that convex geometries of convex dimension 2
may be represented as a set of circles in the plane. Very recently 
Adaricheva and Bolat \cite{adarichevabolat} found an obstruction for
representing any convex geometries with circles. In this section
we present an Erd\H os-Szekeres type obstruction for representing
convex geometries with circles or with ellipses
(in the case of circles it is different from the obstruction
of Adaricheva and Bolat).
These are simple consequences of the following results 
of Pach, T\'oth and Dobbins, Holmsen, Hubard.
We say that a family of convex sets are in convex position if
neither of them is in the convex hull of the others.

\begin{thm}[\cite{pachtoth}]\label{ptsegments}
There is an infinite family of ellipses in the plane such
that any three of them are in convex position but no four are.
\end{thm}

\begin{thm}[\cite{holmsenatall}]\label{holmsencross}
For all integers $n > k \ge 1$, there exists a minimal positive integer
$h_k(n)$ such that the following holds: Any family of at least $h_k(n)$ convex bodies in
the plane such that any two have at most $2k$ common supporting lines and any
$m_k$ are in convex position contains $n$ members which are in convex position, where
$m_1=3$, $m_2 = 4$, and $m_k=5$ for all $k\ge 3$.
\end{thm}

\begin{thm}[\cite{holmsenatall}]\label{holmsenset}
There exist arbitrarily large families of convex
sets in the plane such that

$\bullet$ any two members have precisely six common supporting lines,

$\bullet$ any four members are in convex position, and

$\bullet$ no five members are in convex position.

\end{thm}

Tthese results immediately gives the following

\begin{cor}
{\rm a)} There are convex geometries not representable by circles. 

{\rm b)} There are convex geometries not representable by ellipses. 
\end{cor}
\begin{proof}
a) consider the convex geometry defined by at least $h_1(4)$ ellipses of Theorem \ref{ptsegments}. By Theorem \ref{holmsencross}, this geometry cannot be represented
with circles in the plane.

b) Consider the convex geometry defined by at least $h_2(5)$ sets of Theorem \ref{holmsenset}. It is well known that two ellipses can have at most four common supporting lines, so
by Theorem \ref{holmsencross}, this geometry cannot be represented with ellipses in the plane.
\end{proof}

\section{Convex dimension and common supporting lines}

In this section we shall bound the convex dimension of a family of convex sets
by the number of common supporting lines of the pairs.
This will give an obstruction for representing convex geometries with 
families of convex sets with bounded number of common supporting lines.

Let ${\mathscr K}=\{K_1,\ldots,K_n\}$ be a family of compact convex sets
in the plane such that any two of them have at most $k$ common supporting lines.
If for a unit vector $x$ the values of the support functions $\{h(K_i,x)\}_{i=1}^n$
are all distinct then they define an ordering of the sets via
$K_i\prec_x K_j$ iff $h(K_i,x)<h(K_j,x)$. In this case we say that $x$ is regular.

\begin{figure}[h]
\centering
\includegraphics{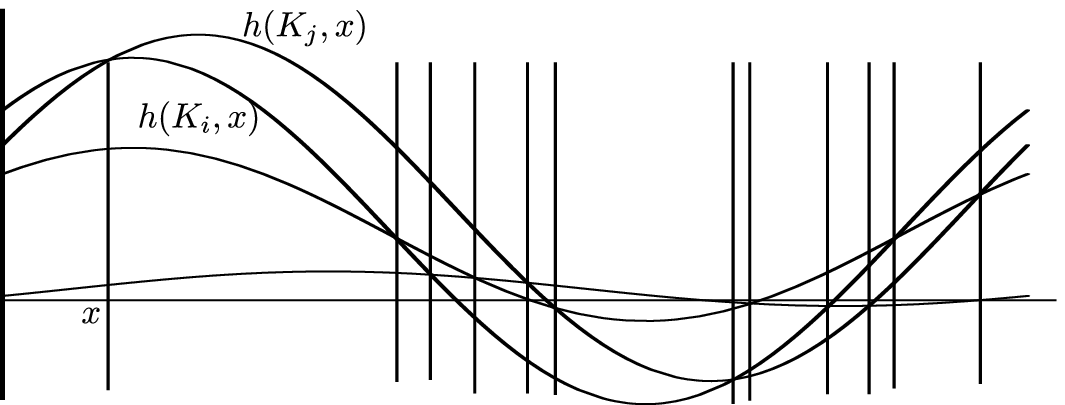}
\caption{}\label{crossing}
\end{figure}

If $x$ is not a regular unit vector then there are $i\not= j$ with $h(K_i,x)=h(K_j,x)$
which gives a common supporting line of $K_i$ and $K_j$.
The non regular vectors correspond to the crossings of the graphs
of the support functions (see Fig. \ref{crossing}).
By the condition, there are altogether at most $k{n\choose2}$
non regular unit vectors. These vectors divide the unit circle into
intervals and for the points of each interval the induced ordering is the same
and these orderings determine the convex geometry  $({\mathscr K},\conv_{\mathscr K})$.
For, let $K\in {\mathscr K}$ and ${\mathscr S}\subseteq {\mathscr K}$.
Then $K\not\in \conv_{\mathscr K}{\mathscr S}$ iff there is an interval $I$
of the unit circle such that $\max_{S\in{\mathscr S}}\{h(S,x)\}<h(K,x)$
for all $x\in I$ and clearly $I$ contains a regular vector. 
Thus we obtained the following result.

\begin{lem}\label{kcrossbound}
If ${\mathscr K}$ is a family of compact convex sets
in the plane such that any two of them have at most $k$ common supporting lines then
\[\cdim({\mathscr K},\conv_{\mathscr K})\le k{n\choose2}\]
\end{lem}

Now we shall construct a convex geometry with large convex dimension.
For this we shall use a result of Edelman and Saks \cite{edelmansaks} which express the 
convex dimension in terms of copoints.
In a convex geometry $(E,\cal C)$ a convex set $A\subseteq E$ is a copoint if there exists
a point $x\in E$ such that $x\not\in A$ and $A$ is
inclusion-maximal among all convex subsets of $E$ that do not contain $x$.
Let $M({\cal C})$ be the set of the copoints of $E$ and $w(M({\cal C}))$ 
be the maximal number of inclusion-incomparable elements of $M({\cal C})$.

\begin{thm}[\cite{edelmansaks}]\label{cdimcopoint} $\cdim(E,{\cal C})=w(M({\cal C}))$
\end{thm}

\noindent
{\bfseries\sffamily The construction:}
Let $\{e_i\}_{i=1}^n$ be the standard basis in ${\mathbb R}^n$ and consider the point set
${CR}=\{0,\pm e_i\}$ and the convex geometry $(CR,\conv_{{\mathbb R}^n})$.
The points $\{\pm e_i\}_{i=1}^n$ are the vertices of the
crosspolytope (generalized octahedron, see Fig.~\ref{crosspoly} for $n=3$) (\cite{ziegler}). 

\begin{figure}[h]
\centering
\includegraphics{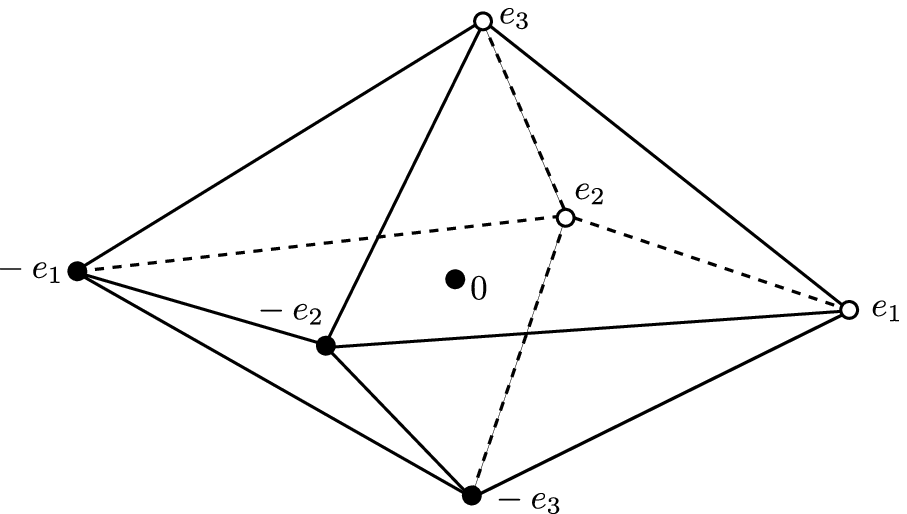}
\caption{}\label{crosspoly}
\end{figure}

The copoint of the point $e_i$ or $-e_i$ is clearly its complement.
Let $A$ be a copoint of $0$. On the one hand the points $e_i$ and $-e_i$
cannot belong to $A$ at the same time. 
On the other hand for any choice of $\varepsilon_i=\pm1$
the set of points $\{\varepsilon_i e_i\}_{i=1}^n$ form
a convex set of our geometry (they are the vertices of a hyperface of the crosspolytope)
not containing $0$, so these sets are the copoints of $0$. Each of them has
$n$ elements so they are inclusion-incomparable. Applying Theorem \ref{cdimcopoint}
we have that
\begin{equation}\label{cr}
\cdim(CR,\conv_{{\mathbb R}^n})=2^n.
\end{equation}

Lemma \ref{kcrossbound} and \eqref{cr} gives the following result. 
\begin{thm}
There is no integer $k_0$ such that each convex geometry can be represented with
a family of convex sets in the plane such that the sets have pairwise at most $k_0$ 
common supporting lines.
\end{thm}

Remark that this theorem yields further obstruction for representing convex geometries
with ellipses in the plane.

\section{Ellipsoids in higher dimensions}

Adaricheva and Cz\'edli \cite{adarichevabolat} raised the question
whether every finite convex geometry can be represented 
with balls in ${\mathbb R}^n$. In this section we prove a weaker statement, the
representation by ellipsoids. We remark that any finite family of ellipsoids
in ${\mathbb R}^n$ determine a convex geometry. The proof goes the same way as
of Lemma \ref{fincross} but now we use the fact that if the support functions of two
ellipsoids are equal on an open set of unit vectors then the two ellipsoids
are equal.

\begin{thm}\label{main1}
Any finite convex geometry with convex dimension $d$ can represented in
$\mathbb R^d$ with ellipsoids which are arbitrary close to a ball.
\end{thm}

\begin{proof}
We shall use ellipsoids in a very special position. For the real numbers
$a_1,\ldots,a_d$ consider the ellipsoid
\[E(a_1,\ldots,a_d)=\left\{(x_1,\ldots,x_d)\in{\mathbb R}^d\colon
\sum_{i=1}^d \frac{x_i^2}{a_i^2}\le1\right\}\]
The support function of this ellipsoid is 
\[h(E(a_1,\ldots,a_d),(x_1,\ldots,x_d))=\sqrt{\sum_{i=1}^d {a_i^2}{x_i^2}}\]
Starting with an arbitrary number $s>1$ we define the sequence of real numbers
\[f(1)=s,\qquad f(i+1)=\sqrt{\frac{f(i)^2+d-1}{d}}.\]
It is easy to check that
\[f(i)>1,\qquad f(i+1)<f(i),\qquad \lim_{i\to\infty}f(i)=1.\]
Let $G$ be a convex geometry defined by a collection of orderings
$\{\preccurlyeq_i\}_{i=1}^d$. For $g\in G$ let $j_i(g)$ be the
$g$'s place according to the $i^{\text{th}}$ ordering.
We associate to a $g\in G$ an ellipsoid $\Phi(g)$ where
\[\Phi(g)=E(f(d+1-j_1(g)),\ldots,f(d+1-j_d(g))),\]
and let ${\mathscr K}=\{\Phi(g)\colon g\in G\}$.
We prove that the map $\Phi$ is an isomorphism between the convex geometries
$(G,\preccurlyeq_i\}_{i=1}^d)$ and $({\mathscr K},\conv_{\mathscr K})$.

Take a convex set $H\subseteq G$ and $g\not\in H$. Then there is some $i$
such that for all $h\in H$ we have $h\prec_i g$ and therefore 
$j_i(g)>j_i(h)$ which gives that $d+1-j_i(g)<d+1-j_i(h)$
and $f(d+1-j_i(g))>f(d+1-j_i(h))$. This means that all 
ellipsoids $\Phi(h)$ are in the halfspace $\{x_i<f(d+1-j_i(g))\}$
so is their convex hull, but this halfspace does not contain $\Phi(g)$.
We get that $\Phi(g)\not\in\conv_{\mathscr K}(\Phi(H))$ and so
$\Phi(H)$ is convex in ${\mathscr K}$. 

Conversely, suppose $H\subseteq G$ is not convex in $G$.
Then there is a $g\not\in H$ such that for all $i$ there is
an $h_i\in H$ with $g\prec_i h_i$. Thus $j_i(g)+1\le j_i(h_i)$ and
\[f(d+1-j_i(g))<f(d-j_i(g))\le f(d+1-j_i(h_i)).\]
This implies that
\[E_i=E(1,\ldots,f(d-j_i(g))),\ldots,1)\subseteq \Phi(h_i).\]
We prove that $\Phi(g)\subseteq\conv_{\mathbb R^d}(\cup_i E_i)$
Using the properties (a), (b), (c) of the support function, we must prove that
$h(\Phi(g),x)\le\max_i\{h(E_i,x)\}$ for all $x$. All the support functions are
positive so enough to consider its square. For an arbitrary 
$x=(x_1,\ldots,x_d)\in {\mathbb R}^d$, $\sum x_i^2=1$
\begin{align*}
h(\Phi(g),x)^2&=\sum_i f(d+1-j_i(g))^2x_i^2=1+\sum_i (f(d+1-j_i(g))^2-1)x_i^2\\
&\le 1+\max_i \{d(f(d+1-j_i(g))^2-1)x_i^2\}
=1+\max_i \{(f(d-j_i(g))^2-1)x_i^2\}\\
&=\max_i \Big\{f(d-j_i(g))^2x_i^2+\sum_{k\not=i}x_k^2\Big\}\\
&=\max_i\{h(E_i,x)^2\}
\end{align*}
In the second row we used that $d(f(d+1-j_i(g))^2-1)=f(d-j_i(g))^2-1$
which comes from the definition of the sequence.
Thus we get that 
$\Phi(g)\subseteq\conv_{\mathbb R^d}(\cup_i E_i)\subseteq\conv_{\mathbb R^d}(\cup_i \Phi(h_i))$
and we know that $\Phi(g)\not\in\Phi(H)$ which gives that $\Phi(H)$ is not convex in
${\mathscr K}$.

Each ellipsoid of the construction contains the unit ball and is contained
in the ball with center the origin and radius $s$ so they are close
to to the unit ball if $s$ sufficiently close to 1.
\end{proof}

\bigskip
\leftline{J. Kincses}
\leftline{University of Szeged, Bolyai Institute}
\leftline{Aradi v\'ertan\'uk tere 1., H-6720 Szeged, Hungary}
\leftline{e-mail: kincses@math.u-szeged.hu}

\begin{thebibliography}{00}

\bibitem{adarichevanation}
Adaricheva, K.; Nation, J.B.: Convex geometries. In Lattice Theory: Special Topics and
Applications, volume 2, G. Grätzer and F. Wehrung, eds., Birkhäuser, 2015.

\bibitem{adarichevabolat}
K. Adaricheva, M. Bolat, Representation of convex geometries by circles on the plane,
arXiv:1609.00092v1.

\bibitem{czedli}
G. Cz\'edli, Finite convex geometries of circles. Discrete Mathematics, 330:61-75, 2014.

\bibitem{czedlikincses}
G. Cz\'edli, J. Kincses, Representing convex geometries by almost-circles, arXiv:1608.06550

\bibitem{holmsenatall}
 M.G. Dobbins, A. Holmsen, A. Hubard, {\it Regular systems of paths and families of convex sets in convex position}, Trans. Amer. Math. Soc. 368 (2016), no. 5, 3271-3303. 

\bibitem{edelmanjamison}
P. H. Edelman and R. E. Jamison, {\it The theory of convex geometries}, Geometriae Dedicata,
19:247-270, 1984.

\bibitem{edelmansaks}
P. H. Edelman and M. E. Saks, {\it Combinatorial representation and convex dimension of convex
geometries}, Order, 5:23-32, 1988.

\bibitem{gruber}
P. M. Gruber, Aspects of approximation of convex bodies, Handbook of convex geometry, Vol. A, B, 319-345, North-Holland, Amsterdam, 1993.

\bibitem{kashiwabara}
K. Kashiwabara, M. Nakamura, and Y. Okamoto. The affine representation theorem for
abstract convex geometries. Comput. Geom., 30(2):129-144, 2005.

\bibitem{kortelovasz}
B. Korte, L. Lov\'asz, R. Schrader, Greedoids, Springer-Verlag, Berlin, 1991.

\bibitem{pachtoth}
J. Pach and G. T\'oth, {\it Erd\H os-Szekeres-type theorems for segments and noncrossing convex
sets}, Geom. Dedicata 81 (2000), no. 1-3, 1?12,

\bibitem{richterrogers}
M. Richter, L. G. Rogers, {\it Embedding convex geometries and a bound on convex
dimension}, Discrete Mathematics; accepted subject to minor changes; arXiv:1502.01941

\bibitem{schneider}
R. Schneider, Convex Bodies: The Brunn-Minkowski Theory, Encyclopedia Math. Appl., vol. 44, Cambridge University Press, Cambridge, 1993.

\bibitem{ziegler}
G.M. Ziegler, Lectures on Polytopes, in: Graduate Texts in Mathematics, vol. 152, Springer-Verlag, New York, 1995.
\end{thebibliography}
\end{document}